\documentclass[a4paper]{amsart}

\providecommand{\meantmp}[2]{#1({#2}#1)}
\providecommand{\mean}[1]{\meantmp{}{#1}}

\providecommand{\areatmp}[2]{#1\langle{#2}#1\rangle}
\providecommand{\area}[1]{\areatmp{}{#1}}

\usepackage{amssymb}
\usepackage{amsmath}
\usepackage{amssymb}
\usepackage{amsthm}
\usepackage{bbm}
\usepackage{bm}
\usepackage{mathrsfs}
\usepackage{color}

\usepackage{graphicx}

\usepackage{tikz}
\usetikzlibrary{patterns}

\usepackage[all]{paper_diening}

\numberwithin{equation}{section}

\newcommand{\bv}{\setBV}

\newcommand{\mres}{\mathbin{\vrule height 1.6ex depth 0pt width
0.13ex\vrule height 0.13ex depth 0pt width 1.3ex}}

\newcommand{\dif}{\,\mathrm{d}}

\newcommand{\N}{\setN}
\newcommand{\R}{\setR}
\newcommand{\locc}{\loc}

\newcommand{\ball}{\mathrm{B}}

\newcommand{\sobo}{\mathrm{W}}
\newcommand{\lebe}{\mathrm{L}}
\newcommand{\hold}{\mathrm{C}}
\newcommand{\D}{D} 

\providecommand{\whitneybox}{
  \draw[dashed] (-0.1,-0.1) rectangle (1.1,1.1);
  \draw[dotted] (0.1,0.1) rectangle (0.9,0.9);
}
\providecommand{\whitneyboxtwo}{
  \fill[pattern=north west lines,even odd rule ] (-0.1,-0.1) rectangle (1.1,1.1)
 (0.1,0.1) rectangle (0.9,0.9);
}

\begin{document}
\title[Lipschitz truncation on $\setBV$]{The Lipschitz truncation of\\
  functions of bounded variation}
\author[D.~Breit]{Dominic Breit}
\author[L.~Diening]{Lars Diening}
\author[F.~Gmeineder]{Franz Gmeineder}
\thanks{The authors are grateful to the Edinburgh Mathematical Society
  for financial support.}


\begin{abstract}
  We construct a Lipschitz truncation which approximates functions of
  bounded variation in the area-strict metric. The Lipschitz
  truncation changes the original function only on a small set similar
  to Lusin's theorem.  Previous results could only give estimates
  on the Lebesgue measure of the set where the Lipschitz
  approximations differ from the original function.
\end{abstract}

\maketitle

\section{Introduction}
\label{sec:introduction}

It is a classical fact attributable to \textsc{Lusin} \cite{Lus12}
that any $u\in\lebe^{p}(\Omega)$ with~$\Omega\subset \Rn$ a bounded, open
set, $1\leq p < \infty$, can be approximated by continuous functions
$u_{\lambda}$ such that $u$ is only changed on a small
set, i.e.
\begin{align}
  \norm{u-u_\lambda}_{p} &\to 0 \qquad \text{and} \qquad
  \mathscr{L}^{n}(\{u\neq u_{\lambda}\}) \to 0
\end{align}
as $\lambda \to \infty$.  Here, $\mathscr{L}^n$ is the Lebesgue measure
on~$\Rn$. This was extended by
\textsc{Liu}~\cite{Liu77} to Sobolev functions, showing that for 
any~$u \in \sobo^{1,p}(\Omega)$ one can find $C^1$-appoximations 
$u_\lambda$ such that
\begin{align}
  \norm{u-u_\lambda}_{1,p} &\to 0 \qquad \text{and} \qquad
  \mathscr{L}^{n}(\{u\neq u_{\lambda}\}) \to 0
\end{align}
as~$\lambda \to \infty$. This is called \emph{Lusin property for
  Sobolev functions}.

A qualitative version thereof has been introduced by \textsc{Acerbi \&
  Fusco}: As established in \cite{AceFus84,AceFus88}, for any
$u\in\sobo^{1,p}(\Omega)$ with $1\leq p< \infty$ and all $\lambda>0$
there exist Lipschitz functions $u_{\lambda}$ such that
\begin{align}\label{eq:AcerbiFusco}
  \norm{\nabla u_\lambda}_\infty &\leq c\, \lambda,
  \qquad \text{and} \qquad
  \mathscr{L}^{n}(\{u\neq u_{\lambda}\}) \leq
  \frac{c\, \norm{u}_{1,p}^p}{\lambda^p}
\end{align}
with $c$ independent of~$u$. It is possible to improve the second
bound to
\begin{align}
  \label{eq:evans-sobolev}
  \mathscr{L}^{n}(\{u\neq u_{\lambda}\}) \leq
  \frac{\delta_u(\lambda)\, \norm{u}_{1,p}^p}{\lambda^p}
\end{align}
with $\delta_u(\lambda) \to 0$ (depending on~$u$), cf. \textsc{Evans
  \& Gariepy} \cite[Chapter~6.6.3, Thm.~3]{EvaGar92}.  Again, this
implies $\norm{u-u_\lambda}_{1,p} \to 0$.

Since it is of class $\sobo^{1,\infty}$ and coincides with $u$ apart
from a set of small Lebesgue-measure, $u_{\lambda}$ is usually
referred to as \emph{Lipschitz truncation}. It is a core feature that
$u_{\lambda}$ differs from $u$ only on a small set. This cannot be
achieved by plain mollification.

The Lipschitz truncation has numerous applications in the calculus of
variations~\cite{AceFus87,DieLenStrVer12}, regularity
theory~\cite{Lew93,CarFusMin98,BulDieSch16}, existence of weak
solutions~\cite{FreMS03,DieMS08,BreDieFuc12,SulTsch19,Zhan88} just to name a
few.

For Lipschitz domains it is possible to preserve zero boundary data of
Sobolev functions, see~\cite{Lan96}. It is possible to
obtain additionally stability of the mapping~$u \mapsto u_\lambda$ in all
Lebesgue spaces, see~\cite{BreDieFuc12,DieKreSul13}.

The Lipschitz truncation has been extended partially to functions of
bounded variation $u$ by \textsc{Evans \& Gariepy}
\cite[Chapter.~6.6.2,~Thm.~2]{EvaGar92}, estabilishing the existence of
Lipschitz approximations~$u_\lambda$ such~\eqref{eq:AcerbiFusco} holds.
However, the corresponding substitute for~\eqref{eq:evans-sobolev}
\begin{align}
  \label{eq:evans-bv}
  \mathscr{L}^{n}(\{u\neq u_{\lambda}\}) \leq
  \frac{\delta_u(\lambda)\, \norm{u}_{\setBV(\Omega)}}{\lambda}
\end{align}
cannot be true for~$\setBV$-functions. In fact, this
and~\eqref{eq:AcerbiFusco} would imply $u_\lambda \to u$
in~$\setBV(\Omega)$ and thereby yield the contradictory denseness of
Lipschitz functions in~$\setBV(\Omega)$ for the norm topology; note
that the respective closure of Lipschitz functions is
$W^{1,1}(\Omega)$. In consequence, $\|u-u_{\lambda}\|_{\bv}\to 0$
\emph{cannot hold} for arbitrary $\bv$-functions.

The goal of this paper hence is to extend the Lipschitz truncation
technique to the setting of~$\setBV(\Omega)$ with~$u_\lambda \to u$ in
a useful metric, necessarily weaker than the norm topology. One
possibility is the notion of weak* convergence. However, this notion
is too weak for many aspects.  A more useful concept is the one of
strict convergence, which requires that additionally the total
variation converges, i.e.
$\abs{Du_\lambda}(\Omega) \to \abs{Du}(\Omega)$.

A slight but effective modification of strict convergence is the
area-strict convergence, since it is more flexible in the
applications. This topology is somehow the strongest one, for which
approximation by smooth functions is still possible. Moreover, the
area-strict convergence (in contrast to weak* convergence) ensures
both continuity of convex functionals with linear growth~\cite{Res68}
and continuity of the trace operator, cf.~\cite{EvaGar92}. For these
reasons we aim for area-strict convergence of our Lipschitz
truncation.

The heart of the Lipschitz truncation is the pointwise estimate
\begin{align}
  |u(x)-u(y)|\leq c|x-y|(\mathcal{M}(D u)(x)+\mathcal{M}(D u)(y))\;\;\text{for}\;\mathscr{L}^{n}\text{-a.e.}\;x,y\in\R^{n},
\end{align}
where $\mathcal{M}$ denotes the usual Hardy-Littlewood maximal
operator, being valid for any Sobolev as well
as~ any $\setBV$-function.  As such, $u$ is Lipschitz continuous on the
closed set
$\mathcal{O}_{\lambda}^\complement:=\{\mathcal{M}(\nabla u)\leq\lambda\}$ (the
\emph{good set}) with Lipschitz constant uniformly proportional to
$\lambda$. Using a suitable extension theorem, it is possible to
modify~$u$ on the \emph{bad set}
$\mathcal{O}_{\lambda}:=\{\mathcal{M}(\nabla u)>\lambda\}$ such that
its modification~$u_\lambda$ is Lipschitz continuous. Among all other extensions\footnote{For some steps the
  McShane and the Kirszbraun extension is sufficient, but both fail
  the useful stability estimates, since constant functions are not
  necessarily extended as constant functions.}, the particular extension based on Whitney coverings of $\mathcal{O}_{\lambda}$ has turned out most suitable. 
For this we pick a Whitney
covering $(Q_{j})_{j\in\mathbb{N}}$ of the bad
set~$\mathcal{O}_{\lambda}$ with a corresponding partition of unity
$(\eta_{j})_{j\in\mathbb{N}}$. Let $\mean{u}_{Q_j}$ denote the mean
value of~$u$ over~$Q_j$.  Then the Lipschitz truncation is usually
defined as
\begin{align}\label{eq:defwl'0}
  u_\lambda &:= u- \sum_{j \in \mathbb{N}} \eta_j (u - \mean{u}_{Q_j})
    =\begin{cases}
      u &\qquad \text{on $\mathcal{O}_\lambda^\complement$},
        \\
      \sum_j \eta_j \mean{u}_{Q_j} &\qquad \text{on $\mathcal{O}_\lambda$}.
    \end{cases}
\end{align}
In particular, we replace~$u$ on the bad set locally by its mean
values to obtain higher regularity. To preserve zero boundary values
one has to replace~$\mean{u}_{Q_j}$ close to the boundary by zero.
However, as we will see in Remark~\ref{rem:failure}, this truncation does not give~$u_\lambda \to u$ in the strict sense, as can be seen from
\begin{align*}
u\,:\, (-1,1)^2 \to \setR,\; x
  \mapsto \sgn(x_2-x_1).
\end{align*}
Here the chief issue is that the jump on the diagonal will turn into a zigzag
isolines, which increases the total variation, cf. Figure~\ref{fig:whitney}.

On the other hand, it is well-known that mollification leads to a
area-strict convergence approximation. However, this would change the
function globally, which is undesired in the applications.  Thus, to
overcome the problems with the classical Lipschitz truncation we
propose a modified Lipschitz truncation based on local corrections
using mollification. To be precise, we define
\begin{align}\label{eq:lipdef0}
  u_\lambda &:= T_\lambda u := u - \sum_{j\in\mathbb{N}} \Big(\eta_j (u-\mean{u}_{Q_j}) - \phi_j * (\eta_j
              (u-\mean{u}_{Q_j})) \Big).
\end{align}
Here, for $j\in\mathbb{N}$, $\phi_j$ denotes a suitable mollifier with
regularisation radius being adapted to the cube $Q_{j}$.

The main feature of the operator~$T_\lambda$ is that it posses a nice
(almost) dual operater~$S_\lambda$ with
\begin{align}
  S_\lambda\rho := \rho - \sum_{j}\eta_{j}(\rho-\varphi_{j}*\rho).
\end{align}
The operator~$S_\lambda$  is non-expansive on~$L^\infty$ and satisfies
nice commutator type estimates, see Lemma~\ref{lem:Slambda}, i.e. 
\begin{align*}
  \bigabs{\skp{D T_\lambda u}{\rho} - \skp{Du}{S_\lambda \rho}}
  &\leq c\, h(\lambda) \abs{Du}(\mathcal{O}_\lambda)\, \norm{\rho}_\infty,
\end{align*}
where $h(\lambda) \to 0$ for~$\lambda \to \infty$.

This technique allows us to construct a modified Lipschitz truncation
with $u_\lambda \to u$ in the area-strict sense, simultaneously being able to
preserve zero boundary data. Our main theorem then reads as follows:
\begin{theorem}\label{thm:main}
  Let $\Omega=\Rn$ or let $\Omega\subset\R^{n}$ be an open, bounded
  Lipschitz domain. Let $h \,:\, (0,\infty) \to (0,1]$ be a
  non-increasing function with
  $\lim_{\lambda \to \infty} h(\lambda) = 0$.  Then there exists a
  constant $c=c(n,\Omega)>0$ such that for any $u\in\setBV(\Omega)$
  and $\lambda>0$ there is $u_{\lambda}\colon\Omega\to\R$ with the
  following properties:
  \begin{enumerate}
  \item\label{itm:thm-lip} \emph{(Lipschitz property)} $u_{\lambda}\in
    \sobo^{1,\infty}(\Omega)$ together with $\|\nabla
    u_{\lambda}\|_{\lebe^{\infty}(\Omega)}\leq \frac{c\,\lambda}{h(\lambda)^{n+1}}$. 
  \item\label{itm:thm-small} \emph{(Small change)} We have
    $\set{u_\lambda \neq u}
    \subset\mathcal{O}_{\lambda}:=\{\mathcal{M}(Du)>\lambda\}$ and
    \begin{align*}
      \mathscr{L}^{n}(\mathcal{O}_{\lambda}) \leq c\, \frac{\abs{Du}(\mathcal{O}_\lambda)}{\lambda}.
    \end{align*}
  \item\label{itm:thm-stab} \emph{(Stability)} The mapping $T_{\lambda}\colon u\mapsto u_{\lambda}$ is stable in the sense that for all $1\leq q\leq\frac{n}{n-1}$ there holds
    \begin{align*}
      \|u_{\lambda}\|_{\lebe^{q}(\Omega)}&\leq
      \,c\|u\|_{\lebe^{q}(\Omega)},
      \\
      \|\nabla
      u_{\lambda}\|_{\lebe^{1}(\Omega)} &\leq\,c |\D u|(\Omega). 
    \end{align*}
  \item\label{itm:thm-conv} \emph{(Convergence)} $u_\lambda \to u$ area-strictly
    in $\setBV(\Omega)$ as $\lambda\to\infty$.  More precisely, if
    $\D u = \D^{a}u+\D^{s}u = \nabla u\mathscr{L}^{n} + \D^{s}u$ is
    the Lebesgue-Radon-Nikod\'{y}m decomposition of $\D u$, then
    \begin{align}
      \label{eq:conv}
      \begin{alignedat}{2}
        \indicator_{\mathcal{O}_{\lambda}^\complement}\nabla u_{\lambda}&
        \to \nabla
        u &\qquad&\text{in $L^1(\Omega)$},
        \\
        \nabla u_{\lambda}\mathscr{L}^{n}\mres
        \mathcal{O}_{\lambda}&\to \D^{s}u\mres \Omega
        &&
        \text{area strictly},
      \end{alignedat}
    \end{align}
    as $\lambda\to\infty$.  Moreover,
    \begin{align}
    \label{eq:main-area-strict}
    \area{\D T_{\lambda}u}(\R^{n})
    &\leq \area{Du}(\Rn) + ch(\lambda)|Du|(\mathcal{O}_{\lambda}) +
      \frac{c}{\lambda}\abs{\D u}(\Rn),
      \\
    \label{eq:main-area-strict2}
    \abs{\D T_{\lambda}u}(\R^{n})
    &\leq \abs{Du}(\Rn) + ch(\lambda)|Du|(\mathcal{O}_{\lambda}).
  \end{align}

  \item\label{itm:thm-zero} \emph{(Zero boundary values)} If $u=0$
    on~$\Omega^\complement$, then $u_\lambda = 0$
    on~$\Omega^\complement$ for all $\lambda \geq \lambda_0$, where
    $\lambda_0= \lambda_0(n,\Omega)$.
  \end{enumerate}
\end{theorem}
Property~\ref{itm:thm-conv} tells us precisely where the single parts
of the approximations $u_{\lambda}$ converge to: Namely, the
restriction of the gradients $\D u_{\lambda}$ to the good set
$\mathcal{O}_{\lambda}^\complement$ converge to the absolutely
continuous part $\D^{a}u$, whereas the restrictions to the bad
set~$\mathcal{O}_\lambda$ converge to the singular part $\D^{s}u$.

The function~$h$ with
$\lim_{\lambda \to 0}h(\lambda)=0$ as it appears
in~\eqref{eq:main-area-strict} and \eqref{eq:main-area-strict2}
ensures the area strict convergence. However, in return it appears in
the Lipschitz estimate in~\ref{itm:thm-lip} additionally in the
denominator.

The outline of the paper is as follows. In Section~\ref{sec:prelim} we
collect the requisite background facts on functions of bounded
variation and maximal functions of Radon measures.
Then in Section~\ref{sec:main} we present our Lipschitz truncation for
$\setBV$-functions, which concludes in Subsection~\ref{sec:proof-theor-refthm:m} with the
proof of Theorem~\ref{thm:main}.

\section{Preliminaries}
\label{sec:prelim}

Throughout, $\Omega$ denotes an open subset of $\R^{n}$ with
$n\geq 2$. Given $x\in\R^{n}$ and $r>0$, we denote by
$\ball_r(x):=\{y\in\R^{n}\colon |x-y|<r\}$ the open ball of radius $r$
centered at $x$. Cubes $Q\subset\R^{n}$ are always understood to be
non-degenerate and parallel to the axes, and we denote $\ell(Q)$ their sidelength. The
$n$-dimensional Lebesgue measure is denoted $\mathscr{L}^{n}$ and the
$n-1$-dimensional Hausdorff measure is denoted
by~$\mathscr{H}^{n-1}$. Sometimes we use the notation
$|U|:=\mathscr{L}^{n}(U)$ for a measurable set $U\subset\R^{n}$. We
use~$\indicator_U$ for the indicator function of the set~$U$.

\subsection{Radon measures}
\label{sec:gener-notat-radon}

The space of $\setR^m$-valued Radon measures on $\Omega$ with finite
total variation is denoted $\mathscr{M}(\Omega;\setR^m)$, i.e.,
$|\mu|(\Omega)<\infty$.  Given $\mu\in\mathscr{M}(\Omega;\setR^m)$ or
$u\in\lebe_{\locc}^{1}(\Omega;\R^{m})$ and a measurable subset
$U\subset\Omega$ with $|U|\in (0,\infty)$, we use the notation
\begin{align*}
  \mean{\mu}_{U}:=\dashint_{U}\dif\mu:=\frac{\mu(U)}{|U|},\;\;\;\mean{u}_{U}:=\dashint_{U}u\dif
  x := \frac{1}{|U|}\int_{U}u\dif x.  
\end{align*}
The space~$\mathscr{M}(\Omega;\R^m)$ can be identified with the dual
space of~$C_0(\Omega;\R^m)$.  We say that~$\mu_k$ converges weakly*
to~$\mu$ if $\mu_k$ converges weakly* in the sense of~$(C_0(\Omega;\R^m))^*$.

Let $\mu,\mu_k \in\mathscr{M}(\Omega;\R^{m})$. We say that $\mu_k$
converges strictly to~$\mu$ if $\mu_k$ converges weakly* to $\mu$ and
$\abs{\mu_k}(\Omega) \to \abs{\mu}(\Omega)$.

The notions of weak* convergence and strict convergence are too weak
for some applications. Therefore, we introduce in the following the
concepts of area-strict and $f$-strict convergence.

Any $\mu\in\mathscr{M}(\Omega;\setR^m)$ can be decomposed as
$\mu=\mu^{a}+\mu^{s}$, where $\mu^{a}\ll\mathscr{L}^{n}$ and
$\mu^{s}\bot\mathscr{L}^{n}$. We shall refer to this as
\emph{Lebesgue-Radon-Nikod\'{y}m decomposition} of $\mu$.

Let $f\colon\R^{m}\to\R$ be a convex function of linear
growth, i.e., there exist $c_{f},C_{f}>0$ such that
$c_{f}|z|\leq f(z)\leq C_{f}(1+|z|)$ holds for all
$z\in\R^{m}$. We define its \emph{recession function}
$f^{\infty}\colon\R^{m}\to\R$ by
\begin{align*}
  f^{\infty}(z):=\lim_{t\searrow 0}tf\Big(\frac{z}{t}\Big),\qquad z\in\R^{m}. 
\end{align*}
Given a Radon measure~$\mu \in \mathcal{M}(\Omega; \setR^m)$ we define
the Radon measure~$f(\mu)$ by
\begin{align}
  f(\mu) &:=f\Big(\frac{\dif \mu}{\dif\mathscr{L}^n} \Big) \mathscr{L}^n +
            f^{\infty}\Big(\frac{\dif \mu^s}{\dif\abs{\mu^s}} \Big) |\mu^s|.
\end{align}
Let $\mu,\mu_k \in\mathscr{M}(\Omega;\R^{m})$. We say that $\mu_k$
converges $f$-strictly to~$\mu$ if $\mu_k$ converges weakly* to~$\mu$
and
\begin{align*}
  \abs{f(\mu_k)}(\Omega) \xrightarrow{k \to \infty}  \abs{f(\mu)}(\Omega).
\end{align*}
The case $f(z)=\abs{z}$ recovers the strict convergence.

For $f(z)= \sqrt{\smash{\abs{z}}^2+1}$ we abbreviate~$\area{\mu} := f(\mu)$.  We say
$\mu_k$ converges area-strictly to~$\mu$ if
$\area{\mu_k}(\Omega) \to \area{\mu}(\Omega)$.  Note that the
area-strict convergence of~$\mu_k$ to $\mu$ is equivalent to the
strict convergence of~$(\mu_k, \mathscr{L}^n)$ to
$(\mu, \mathscr{L}^n)$.

\subsection{Functions of bounded variation}\label{sec:BV}

We now collect the background definitions and facts on
$\setBV$-functions, all of which can be traced back to
\cite[Chpt.~5]{EvaGar92} and \cite{AmbFusPal00}. Let $\Omega$ be an open subset of
$\R^{n}$. A measurable function $u\colon \Omega\to\R$ is said to
be of \emph{bounded variation} (in which case we write
$u\in\setBV(\Omega)$) if and only if
$u\in\lebe^{1}(\Omega)$ and its total variation
\begin{align*}
  |\D u|(\Omega):=\sup\left\{\int_{\Omega}u\,\divergence(\varphi)\dif x\colon\;\varphi\in\hold_{c}^{1}(\Omega;\R^{n}),\;|\varphi|\leq 1\right\}
\end{align*}
is finite.  The norm on $\setBV(\Omega)$ is given by
$\|u\|_{\setBV(\Omega)}:=\|u\|_{L^1(\Omega)}+|\D u|(\Omega)$. Convergence with
respect to the norm is referred to as \emph{strong convergence}.

The Lebesgue-Radon-Nikod\'{y}m decomposition of $\D u$ into its
absolutely continuous and singular parts for $\mathscr{L}^{n}$ reads
as
\begin{align}\label{eq:LRN}
  \D u = \D^{a}u + \D^{s}u = \nabla u\mathscr{L}^{n} + \frac{\dif\D^{s}u}{\dif|\D^{s}u|}|\D^{s}u|, 
\end{align}
where $\nabla u$ is the approximate gradient.

Given $u,u_k\in\setBV(\Omega)$, we say that
$u_k$ converges weakly* in $\setBV(\Omega)$
provided $u_k\to u$ in $\lebe^{1}(\Omega)$ and
$\D u_k\stackrel{*}{\rightharpoonup}\D u$ in
$\mathscr{M}(\Omega;\R^{n})$.

While weak* convergence is useful for compactness arguments, it is
insufficient for a variety of other applications.  One often needs the
stronger notion of strict or area strict convergence, which we introduce in the
following.

We say that $u_j$ converges stricly (resp. area strictly or $f$-strictly) to
$u$ if $u_j$ converges to~$u$ in $L^1(\Omega)$ and $Du_j$ converges
strictly to ~$Du$ (resp. area strictly or $f$-strictly), see
Subsection~\ref{sec:gener-notat-radon} for the assumptions on~$f$.

Area-strict convergence implies $f$-strict convergence due to
\textsc{Goffman \& Serrin}~\cite{GofSer64} and
\textsc{Reshetnyak}~\cite{Res68}. Therefore it suffices in this
article to restrict ourselves area strict convergence.
Note that~$u_k \to u$ in $L^1(\Omega)$ implies that
\begin{align}
  \label{eq:L1-lsc}
  f(Du)(\Omega) &\leq \smash{\liminf_{k\to \infty}}
                  f(Du_k)(\Omega),
\end{align}
with equality only if $u_k$ converges to~$u$ in the $f$-strict sense.
Area-strict convergence is in some sense the strongest topology still
allowing for smooth approximation, yet being weaker than the norm
topology.  The single convergences are linked as follows:
\begin{align}\label{eq:linkconvs}
  \text{area-strict convergence}\Longrightarrow\text{strict
  convergence}\Longrightarrow\text{weak* convergence}.
\end{align}

\subsection{The Hardy-Littlewood maximal operator for measures}
\label{sec:hardy-littl-maxim}

Let us review the properties of the maximal operator on Radon
measures. For a Radon measure~$\mu$ on~$\Rn$ we define
\begin{align}\label{eq:fractional1}
  \mathcal{M}\mu(x):=\sup_{Q\ni x}\mathcal{M}_{Q}\mu(x):=\sup_{Q\ni
  x}\frac{|\mu|(Q)}{\ell(Q)^{n}},
\end{align}
where the supremum is taken over all cubes.
By the Riesz representation theorem for Radon measures, we may equivalently write 
\begin{align}\label{eq:max}
  (\mathcal{M}\mu)(x)=\sup_{Q\ni x}\sup_{\varphi\in C_0(Q;\R^{m})\setminus\{0\}}\frac{\langle \varphi,\mu\rangle}{\|\varphi\|_{L^{\infty}}\ell(Q)^{n}},\qquad x\in\R^{n}.
\end{align}
For future reference, we collect the most important results of the operator  in 
\begin{lemma}\label{lem:franz}
  The operator $\mathcal{M}$ as defined in \eqref{eq:fractional1} satisfies each of the following:  
  \begin{enumerate}
  \item\label{item:MAX1} For each $\mu\in\mathscr{M}(\R^{n};\R^{m})$, $\mathcal{M}\mu\colon \R^{n}\to \R$ is a lower semicontinuous function. 
  \item\label{item:MAX2} There exists a constant $c=c(n)>0$ such that $\mathscr{L}^{n}(\{\mathcal{M}\mu>\lambda\})\leq \frac{c}{\lambda}|\mu|(\R^{n})$ for all $\mu\in\mathscr{M}(\R^{n};\R^{m})$ and all $\lambda>0$. 
  \item\label{item:MAX3} There exists $c=c(n)>0$ such that for any $v\in\setBV(\R^{n})$ there holds 
    \begin{align*}
      |v(x)-v(y)|& \leq c|x-y|\big(\mathcal{M}\D v(x)+\mathcal{M}\D v(y) \big)
    \end{align*}
    for $\mathscr{L}^{n}$-a.e. $x,y\in\R^{n}$. Here, $v(x)$ and $v(y)$ are understood in the sense of precise representatives.
  \end{enumerate}
\end{lemma}
\begin{proof}
  Items~\ref{item:MAX1} and~\ref{item:MAX2} can be established completely analogous as for the well-known case of $\lebe^{1}$-functions. By \cite[Thm.~2.5]{DeVSha84}, for any $u\in\lebe_{\locc}^{1}(\R^{n})$ there holds 
  \begin{align*}
    |u(x)-u(y)|\leq c|x-y|\bigg( \sup_{Q\ni x}\frac{1}{\ell(Q)}\dashint_{Q}|u-\mean{u}_{Q}|\dif z + \sup_{Q\ni y}\frac{1}{\ell(Q)}\dashint_{Q}|u-\mean{u}_{Q}|\dif z \bigg)
  \end{align*}
  for $\mathscr{L}^{n}$-a.e. $x,y\in\R^{n}$. Now it suffices to apply the Poincar\'{e} inequality for $\setBV$-functions and the definition of $\mathcal{M}$ to conclude the claim. The proof is complete. 
\end{proof}

\section{\texorpdfstring{Lipschitz truncation in $\setBV$}{Lipschitz truncation in BV}}
\label{sec:main}

In this section we present our Lipschitz truncation of
$\setBV$--functions. Let $u \in \setBV(\Rn)$ be
given. If we just have $u \in \setBV(\Omega)$ with
$\mathscr{H}^{n-1}(\partial\Omega)<\infty$, then we can extend it by
zero to all of~$\Rn$ using~\cite[Chapter~5.4, Theorem~1]{EvaGar92}. In this
way, we then obtain that $u$ appears as a restriction of some element
of $\setBV(\R^{n})$.

\subsection{Whitney decomposition of the bad set}
\label{sec:whitn-decomp}

For $\lambda>0$ we define the \emph{bad set}
$\mathcal{O}_{\lambda}:=\{\mathcal{M}(\D u) >\lambda\}$. We decompose
this bad set in a standard way by means of a Whitney cover. For this we
use the version~\cite[Lemma~3.1]{{DieRuzWol10}}.  We can
decompose~$\mathcal{O}_\lambda$ into a countable family of open cubes
$\set{Q_{j}}$, each $Q_{j}$ having side length~$r_j>0$, such that the
following holds:
\begin{enumerate}[label={(W\arabic{*})}]
\item\label{W1} $\bigcup_{j}\frac{1}{2}Q_{j}=\mathcal{O}_\lambda$
\item\label{W2} For all $j\in\mathbb{N}$ we have $8Q_{j}\subset\mathcal{O}_\lambda$ and $16Q_{j}\cap(\R^{n}\setminus\mathcal{O}_\lambda)\neq\emptyset$. 
\item\label{W3} If $Q_{j}\cap Q_{k} \neq\emptyset$, then $\frac{1}{2}r_{k}\leq r_{j} \leq 2r_{k}$. 
\item\label{W4} $\frac{1}{4}Q_{j}\cap\frac{1}{4}Q_{k}=\emptyset$ for all $j\neq k$.
\item\label{W5} At every point at most $120^{n}$ of the sets $4Q_{j}$ intersect. 
\end{enumerate}

Subject to the covering $\set{Q_j}$ there exists a partition of
unity $\set{\eta_j} \subset \hold^\infty_c(\Rn)$ with
\begin{enumerate}[label={\rm (P\arabic{*})}, leftmargin=*]
\item \label{itm:P1} $\indicator_{\frac{1}{2}Q_j}\leq \eta_j\leq
  \indicator_{\frac 34 Q_j}$,
\item \label{itm:P3}
  $\norm{\eta_j}_\infty + r_j \norm{\nabla \eta_j}_\infty + r_j^2
  \norm{\nabla^2 \eta_j}_\infty \leq c$.
\end{enumerate}
For each $k \in \setN$ we define $A_k:= \set{ j \,:\, \frac 34
  Q_k \cap \frac 34 Q_j \neq \emptyset}$. Then
\begin{enumerate}[label={\rm (P\arabic{*})}, leftmargin=*,start=3]
\item \label{itm:P4} $\sum_{j \in A_k} \eta_j = 1$ on $\frac 34 Q_k$.
\end{enumerate}
Moreover, we have the following:
\begin{enumerate}[label={\rm (W\arabic{*})}, leftmargin=*, start=6]
\item \label{itm:whit_fat} If $j \in A_k$, then $\abs{Q_j
    \cap Q_k} \geq 16^{-n} \max \set{\abs{Q_j},
    \abs{Q_k}}$.
\item \label{itm:whit_fat34} If $j \in A_k$, then $\abs{\frac 34 Q_j
    \cap \frac 34 Q_k} \geq  \max \set{\abs{Q_j},
    \abs{Q_k}}$.
\item \label{itm:whit_radii} If $j \in A_k$, then $\frac 12 r_k \leq
  r_j < 2 r_k$.
\item \label{itm:whit_sum} $\# A_k \leq 120^n$.
\end{enumerate}
Finally, we need the following geometric alternatives in the
spirit of~\cite[Lemma~3.2]{DieSchStrVer17}.
\begin{lemma}
  \label{lem:alternatives}
  Let $Q$ be an open cube of side length~$r$. Then  one of the following
  alternatives holds:
  \begin{enumerate}[label={\emph{(A\arabic{*})}}]
  \item \label{itm:alt_small} There exists $k \in \setN$ such that $Q \cap
    \frac 12 Q_k \neq \emptyset$, $8r \leq r_k$ and $Q \subset \frac
    34 Q_k$.
  \item \label{itm:alt_large} For all $j \in\setN$ with $Q \cap \frac
    34 Q_j \neq \emptyset$ there holds $r_j \leq 16 r$ and $\abs{Q_j}
    \leq 8^n \abs{Q_j \cap Q}$. Moreover, $137 Q \cap
    \mathcal{O}_\lambda^\complement \neq \emptyset$.
  \end{enumerate}
\end{lemma}
\begin{proof}
  If there exists $k \in \setN$ such that $Q \cap \frac 12 Q_k \neq
  \emptyset$ and $8r \leq r_k$, then automatically $Q \subset Q_k$.
  Assume now that such a $k$ does not exist. Then for every $l \in
  \setN$ with $Q \cap \frac 12 Q_l \neq \emptyset$, there holds $r_l
  \leq 8r$. Suppose that $Q \cap \frac 34 Q_j \neq \emptyset$. Now let
  $x \in Q \cap \frac 34 Q_j$, then by~\ref{W1} there exists
  $m$ such that $x \in \frac 12 Q_m$. In particular, we have $Q \cap
  \frac 34 Q_j \neq \emptyset$ due to~\ref{W2} and $\frac 12 Q_m \cap \frac 34 Q_j
  \neq \emptyset$, since both sets contain~$x$. Now, our assumption
  and $Q \cap \frac 34 Q_j \neq \emptyset$ imply $r_m \leq 8r$. On
  the other hand, $\frac 12 Q_m \cap \frac 34 Q_j \neq \emptyset$
  and~\ref{W3} imply $r_j \leq 2 r_m$. Thus, $r_j \leq 16 r$.
  Moreover, it follows from $8r \geq r_m$ that $137 Q = (1+17\cdot 8)Q
  \supset 16 Q_m$.  Since $16 Q_m \cap (\setR^{n} \setminus \mathcal{O}_\lambda)
  \neq \emptyset$, we also get $137 Q \cap (\setR^{n} \setminus
  \mathcal{O}_\lambda) \neq \emptyset$. 
  It remains to prove $\abs{Q_j} \leq 8^{n} \abs{Q_j \cap Q}$. If $r
  \leq \frac 18 r_j$, then $Q \subset Q_j$ and the claim follows. If
  $r \geq \frac 18 r_j$, then there exists an open cube $Q'$ with side
  length $\frac 18 r_j$ such that $Q' \subset Q_j
  \cap Q$.  So in this case $\abs{Q_j \cap Q} \geq \abs{Q'} \geq
  8^{-n} \abs{Q_j}$. The proof is complete. 
\end{proof}

For each~$j \in \setN$ define
\begin{align}
  \label{eq:defvj}
  u_{j}&:=
         \begin{cases} \mean{u}_{Q_j} = \dashint_{Q_j} u \dif x
           &\;\text{if}\;\frac 34 Q_j\subset\Omega
           \\
           0&\;\text{otherwise}. 
         \end{cases}
\end{align}
Note that the~$u_j$ depend implicitly on~$\lambda$. However, for the sake of
readability we avoid an extra index.

Similar to \cite[Lem.~23]{DieKreSul13} we obtain the following
estimates for~$u$ on the Whitney cubes.
\begin{lemma}\label{lem:1}
  There exists a constant $c=c(n,\Omega)>0$ such that for all
  $\lambda>0$ and all $j\in\N$ the following holds:
  \begin{enumerate}
  \item\label{item:aux1} We have
    \begin{align*}
      \dashint_{Q_j}\left\vert \frac{u-u_{j}}{r_{j}}\right\vert\dif x
      \leq c\,\frac{\abs{Du}(Q_j)}{\abs{Q_j}} \leq c\, \lambda.
    \end{align*}
  \item\label{item:aux2} If $k\in\N$ is such that $\frac 34 Q_{j}\cap \frac 34 Q_{k}\neq\emptyset$, then 
    \begin{align*}
      |u_k-u_j|\leq \,c\,\dashint_{Q_j}|u-u_{j}|\dif x + c\,\dashint_{Q_k}|u-u_{k}|\dif x. 
    \end{align*}
  \item\label{item:aux3} If $k\in\N$ is such that
    $\frac 34 Q_{j}\cap \frac 34 Q_{k}\neq\emptyset$, then
    $|u_{j}-u_{k}|\leq cr_{j}\lambda$.
  \end{enumerate}
\end{lemma}
\begin{proof}
  Ad~\ref{item:aux1}. By definition of the $u_{j}$'s, cf.~\eqref{eq:defvj},
  either $\frac{3}{4}Q_{j}\subset\Omega$, in which case we have
  \begin{align*}
    \dashint_{Q_{j}}\left\vert \frac{u-u_{j}}{r_{j}}\right\vert\dif x \leq
    c\,\frac{\abs{Du}(Q_j)}{\abs{Q_j}}
  \end{align*}
  by Poincar\'{e}'s inequality. If $\frac 34 Q_{j}\nsubseteq\Omega$ we deduce
  that $|Q_{j}\cap \Omega^\complement|\geq c|Q_j|$ since
  $\Omega$ has Lipschitz boundary. Therefore, by the variant
  of Poincar\'{e}'s inequality given in \cite[Prop.~5.4.1]{EvaGar92},
  \begin{align*}
    \dashint_{Q_j}\left\vert \frac{u-u_{j}}{r_{j}}\right\vert\dif x
    = \dashint_{Q_j}\left\vert \frac{u}{r_{j}}\right\vert\dif x
    \leq 
    c\,\frac{\abs{Du}(Q_j)}{\abs{Q_j}}. 
  \end{align*}
  By \ref{W2}, we have
  $16Q_{j}\cap \mathcal{O}_{\lambda}^\complement\neq\emptyset$ and thus find $z\in \mathcal{O}_{\lambda}^\complement$ as well as $r\leq 32 r_{j}$ such that
  $16Q_{j}\subset \ball_r(z)$. Therefore,
  \begin{align*}
    \frac{\abs{Du}(Q_j)}{\abs{Q_j}} \leq c\,
    \frac{\abs{Du}(16Q_j)}{\abs{16 Q_j}}\leq \,c\,
    \frac{\abs{Du}(\ball_r(z))}{\abs{\ball_r(z)}} \leq c\,
    (\mathcal{M}|Du|)(z) \leq c\lambda.
  \end{align*}
  Ad~\ref{item:aux2}. Under the assumptions of \ref{item:aux2}, we deduce from \ref{itm:whit_fat34} that $c\max\{|\frac 34Q_j|,|\frac 34 Q_k|\}\leq |Q_j\cap Q_k|$. Thus, 
  \begin{align*}
    |u_{j}-u_{k}| & \leq \dashint_{Q_j\cap
                    Q_k}|u-u_j|\dif x +
                    \dashint_{Q_j\cap Q_k}|u-u_{k}|\dif
                    x 
    \\ &\leq c\left( \dashint_{Q_j}|u-u_{j}|\dif x + \dashint_{Q_k}|u-u_{k}|\dif x \right)
  \end{align*}
  which implies the claim. Ad~\ref{item:aux3}. By \ref{itm:whit_radii}, this is an immediate consequence of \ref{item:aux1} and \ref{item:aux2}. The proof is complete.  
\end{proof}

\subsection{Definition of the Lipschitz truncation}
\label{sec:defin-our-lipsch}

In this subsection we introduce a modified Lipschitz truncation. Toward Theorem~\ref{thm:main}, we begin by showing that the standard Lipschitz truncation
for~$W^{1,p}$-functions cannot be employed as it does not yield strict convergence in~$\setBV$. This is the content of the following remark.
\begin{remark}[Failure of the standard Lipschitz truncation]
  \label{rem:failure}
  Let us explain why the standard Lipschitz truncation cannot yield
  strict convergence.  Consider the the function
  $u\,:\, \Omega \to \setR$ with $\Omega = (-1,1)^2$ and
  \begin{align*}
    u(x) &:= \sgn(x_2-x_1).
  \end{align*}
  Then $u \in \setBV(\Omega)$ and $\abs{Du}(\Omega) =
  2\,\sqrt{2}$. Let
  $\mathcal{O}_\lambda= \set{\mathcal{M}(Du)>\lambda}$ denote the bad
  set. Then for large~$\lambda$, the set~$\mathcal{O}_\lambda$ and its
  Whitney decomposition looks roughly as in Figure~\ref{fig:whitney}. The standard Lipschitz truncation is defined as
  \begin{align*}
    u_\lambda &:= u  + \sum_j \eta_j (\mean{u}_{Q_j}-u)
    =\begin{cases}
      u &\qquad \text{on $\mathcal{O}_\lambda^\complement$},
        \\
      \sum_j \eta_j \mean{u}_{Q_j} &\qquad \text{on $\mathcal{O}_\lambda$}.
    \end{cases}
  \end{align*}
  The dyadic structure of the Whitney cubes forces the isolines
  of~$u_\lambda$ (for large~$\lambda$) to form a zigzag pattern, thereby increasing the length of the isolines. Hence, the co-area formula
  \begin{align*}
    \abs{Du}(\Omega) &= \int_{-\infty}^\infty \mathscr{H}^{n-1}(u^{-1}(\set{t}))\,dt
  \end{align*}
  shows that $\abs{Du_\lambda}(\Omega) \geq (1+\delta)
  \abs{Du}(\Omega)$ for some fixed~$\delta>0$.
  Thus, $u_\lambda$ cannot strictly converge to~$u$ in~$\setBV(\Omega)$.
  \begin{figure}[ht!]
  \begin{tikzpicture}[scale=0.8]
    \begin{scope}[shift={(-2,-2)}]
      \whitneyboxtwo
      \node at (0.5,0.5) {\scalebox{0.8}{$0$}};
    \end{scope}
    \begin{scope}[shift={(-1,-1)}]
      \whitneyboxtwo
      \node at (0.5,0.5) {\scalebox{0.8}{$0$}};
    \end{scope}
    \begin{scope}
      \whitneyboxtwo
      \node at (0.5,0.5) {\scalebox{0.8}{$0$}};
    \end{scope}
    \begin{scope}[shift={(1,1)}]
      \whitneyboxtwo
      \node at (0.5,0.5) {\scalebox{0.8}{$0$}};
    \end{scope}
    \begin{scope}[shift={(2,2)},scale=1]
      \whitneyboxtwo
      \node at (0.5,0.5) {\scalebox{0.8}{$0$}};
    \end{scope}
    \begin{scope}[shift={(3,3)},scale=1]
      \whitneyboxtwo
      \node at (0.5,0.5) {\scalebox{0.8}{$0$}};
    \end{scope}

    \begin{scope}[shift={(-2,-1)},scale=1]
      \whitneybox
      \node at (0.5,0.5) {\scalebox{0.8}{$1$}};
    \end{scope}
    \begin{scope}[shift={(-1,0)},scale=1]
      \whitneybox
      \node at (0.5,0.5) {\scalebox{0.8}{$1$}};
    \end{scope}
    \begin{scope}[shift={(0,-1)},scale=1]
      \whitneybox
      \node at (0.5,0.5) {\scalebox{0.8}{$-1$}};
    \end{scope}
    \begin{scope}[shift={(0,1)},scale=1]
      \whitneybox
      \node at (0.5,0.5) {\scalebox{0.8}{$1$}};
    \end{scope}
    \begin{scope}[shift={(2,3)},scale=1]
      \whitneybox
      \node at (0.5,0.5) {\scalebox{0.8}{$1$}};
    \end{scope}

    \begin{scope}[shift={(3,2)},scale=1]
      \whitneybox
      \node at (0.5,0.5) {\scalebox{0.8}{$-1$}};
    \end{scope}
    \begin{scope}[shift={(1,0)},scale=1]
      \whitneybox
      \node at (0.5,0.5) {\scalebox{0.8}{$-1$}};
    \end{scope}
    \begin{scope}[shift={(-1,-2)},scale=1]
      \whitneybox
      \node at (0.5,0.5) {\scalebox{0.8}{$-1$}};
    \end{scope}

    \begin{scope}[shift={(1.5,3)},scale=0.5]
      \whitneybox
      \node at (0.5,0.5) {\scalebox{0.6}{$1$}};
    \end{scope}
    \begin{scope}[shift={(-.5,1)},scale=0.5]
      \whitneybox
      \node at (0.5,0.5) {\scalebox{0.6}{$1$}};
    \end{scope}
    \begin{scope}[shift={(-1.5,0)},scale=0.5]
      \whitneybox
      \node at (0.5,0.5) {\scalebox{0.6}{$1$}};
    \end{scope}
    \begin{scope}[shift={(1,-.5)},scale=0.5]
      \whitneybox
      \node at (0.5,0.5) {\scalebox{0.6}{$-1$}};
    \end{scope}

    \begin{scope}[shift={(1,2)},scale=1]
      \whitneybox
      \node at (0.5,0.5) {\scalebox{0.8}{$1$}};
    \end{scope}
    \begin{scope}[shift={(2,1)},scale=1]
      \whitneybox
      \node at (0.5,0.5) {\scalebox{0.8}{$1$}};
    \end{scope}
    \begin{scope}[shift={(.5,2)},scale=0.5]
      \whitneybox
      \node at (0.5,0.5) {\scalebox{0.6}{$1$}};
    \end{scope}
    \begin{scope}[shift={(2,.5)},scale=0.5]
      \whitneybox
      \node at (0.5,0.5) {\scalebox{0.6}{$-1$}};
    \end{scope}
    \begin{scope}[shift={(3,1.5)},scale=0.5]
      \whitneybox
      \node at (0.5,0.5) {\scalebox{0.6}{$-1$}};
    \end{scope}
    \begin{scope}[shift={(0,-1.5)},scale=0.5]
      \whitneybox
      \node at (0.5,0.5) {\scalebox{0.6}{$-1$}};
    \end{scope}

    \draw[red] (-1.6,-1.6) -- (3.6,3.6);
    \draw[red,dashed] (-2.6,-0.4) -- (2.6,4.8);
    \draw[red] (-1.6,-1.6) -- (3.6,3.6);
    \draw[red,dashed] (-0.3,-2.5) -- (5,2.8);
    \draw[black] (0.25,0.95) -- (-1.1,4);
    \node[black] at (-1.1,4.15) {\text{$\support(\nabla u_{\lambda})$}};
  \end{tikzpicture}
  \begin{tikzpicture}[scale=0.8]
    \begin{scope}[blue!80!yellow,shift={(-0.14,0.14)}];
      \draw[rounded corners=1.5mm]
      (-2.0,-1.0) --
      (-1.0,-1.0) -- (-1.0,0.0) --
      (0.0,0.0) -- (0.0,1.0) --
      (1.0,1.0) -- (1.0,2.0) --
      (2.0,2.0) -- (2.0,3.0) --
      (3.0,3.0) -- (3.0,4.0);
    \end{scope}
    \begin{scope}[blue!50!yellow,shift={(-0.08,0.08)}];
      \draw[rounded corners=1.5mm]
      (-2.0,-1.0) --
      (-1.0,-1.0) -- (-1.0,0.0) --
      (0.0,0.0) -- (0.0,1.0) --
      (1.0,1.0) -- (1.0,2.0) --
      (2.0,2.0) -- (2.0,3.0) --
      (3.0,3.0) -- (3.0,4.0);
    \end{scope}
    \begin{scope}[blue!30!yellow,shift={(-0.02,0.02)}];
      \draw[rounded corners=1.5mm]
      (-2.0,-1.0) --
      (-1.0,-1.0) -- (-1.0,0.0) --
      (0.0,0.0) -- (0.0,1.0) --
      (1.0,1.0) -- (1.0,2.0) --
      (2.0,2.0) -- (2.0,3.0) --
      (3.0,3.0) -- (3.0,4.0);
    \end{scope}
    \begin{scope}[blue!80!yellow,shift={(0.14,-0.14)}];
      \draw[rounded corners=1.5mm]
      (-1.0,-2.0) --
      (-1.0,-1.0) -- (0.0,-1.0) --
      (0.0,0.0) -- (1.0,0.0) --
      (1.0,1.0) -- (2.0,1.0) --
      (2.0,2.0) -- (3.0,2.0) --
      (3.0,3.0) -- (4.0,3.0);
    \end{scope}
    \begin{scope}[blue!50!yellow,shift={(0.08,-0.08)}];
      \draw[rounded corners=1.5mm]
      (-1.0,-2.0) --
      (-1.0,-1.0) -- (0.0,-1.0) --
      (0.0,0.0) -- (1.0,0.0) --
      (1.0,1.0) -- (2.0,1.0) --
      (2.0,2.0) -- (3.0,2.0) --
      (3.0,3.0) -- (4.0,3.0);
    \end{scope}
    \begin{scope}[blue!30!yellow,shift={(0.02,-0.02)}];
      \draw[rounded corners=1.5mm]
      (-1.0,-2.0) --
      (-1.0,-1.0) -- (0.0,-1.0) --
      (0.0,0.0) -- (1.0,0.0) --
      (1.0,1.0) -- (2.0,1.0) --
      (2.0,2.0) -- (3.0,2.0) --
      (3.0,3.0) -- (4.0,3.0);
    \end{scope}
    \draw[red] (-1.6,-1.6) -- (3.6,3.6);
    \draw[red,dashed] (-2.6,-0.4) -- (2.6,4.8);
    \draw[red] (-1.6,-1.6) -- (3.6,3.6);
    \draw[red,dashed] (-0.3,-2.5) -- (5,2.8);
  \end{tikzpicture}
  \caption{Example of Remark~\ref{rem:failure}. Left: The picture
    shows a zoom of the Whitney cover at the diagonal. The numbers
    indicate the mean values~$\mean{u}_{Q_j}$. The function~$u_\lambda$ is
    locally constant outside the shaded region.  Right: The picture
    shows the resulting isolines of~$u_\lambda$.}

\label{fig:whitney}
\end{figure}
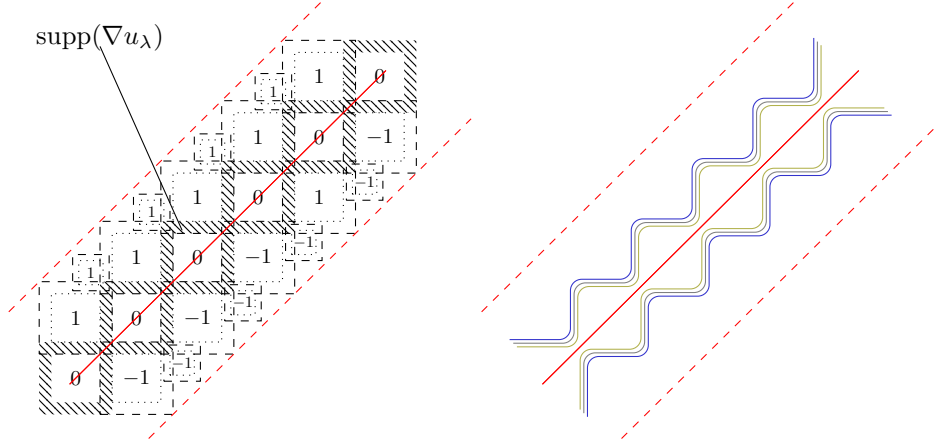
\end{remark}

On the other hand, it is well-known that mollification of a~$\setBV$
function yields a strictly convergent approximation. However, this
approximation would differ in general from the original function
almost everywhere. Therefore, we combine the standard Lipschitz
truncation with a local mollification to obtain our new Lipschitz
truncation converging even in the area-strict sense.

For this purpose, let $h \,:\, (0,\infty) \to (0,1]$ be a
non-increasing function with
$\lim_{\lambda \to \infty} h(\lambda) = 0$. Let $\phi$ be a smooth,
non-negative, radially symmetric mollifier with support in the unit
ball.  For $j\in\mathbb{N}$ let
\begin{align}
  \label{eq:def_phi_j}
  \phi_j(x) &:= (\epsilon_j)^{-n}
              \phi(x/\epsilon_j) \qquad \text{with}
              \qquad \epsilon_j := h(\lambda)\, \tfrac 14
              \,r_j. 
\end{align}
In particular, the supports of the $\phi_j$ shrink faster than the
cubes~$Q_j$ by the factor of~$h(\lambda)$.
Furthermore, we define
\begin{align}
  \label{eq:defBj}
  \mathcal{B}_j u&:= \eta_j (u-u_j) - \phi_j * (\eta_j
                   (u-u_j)).
\end{align}
Then we have $\support(\mathcal{B}_j) \subset Q_j$ due to
$\support(\eta_j) \subset \tfrac 34 Q_j$ and the choice
of~$\epsilon_j$. We now define our truncation operator~$T_\lambda$ by
\begin{align}\label{eq:lipdef}
  T_\lambda u := u_\lambda &:= u - \sum_j \mathcal{B}_j u.
\end{align}
The special choice of this truncation operator will become clear later
when we consider its (almost) dual operator in
Lemma~\ref{lem:Slambda}. As we will see, the map $T_{\lambda}u$ defines an
element in $\setBV(\Rn)$, cf. Lemma~\ref{lem:Bj-stability}.

\subsection{Properties of the Lipschitz truncation}
\label{sec:prop-lipsch-trunc}

In this subsection we study important properties of the Lipschitz
truncation~$T_\lambda u$. We begin with the stability estimates.
\begin{lemma}[Stability]\label{lem:Bj-stability}
  There exists a constant $c=c(n)>0$ such that we have the following $\lebe^{1}$- and $\setBV$-stability estimates for all $1\leq q \leq \frac{n}{n-1}$ and all $j\in\mathbb{N}$:
  \begin{align}\label{eq:stab1}
    \|\mathcal{B}_j u\|_q \leq c\,\int_{Q_j}|u|\dif
    x\hspace{0.5cm}\text{and}\hspace{0.5cm}|D (\mathcal{B}_j u)|(\R^n)
    \leq c\,|Du|(Q_j).  
  \end{align}
  The sum $\mathcal{B}_\lambda u := \sum_j \mathcal{B}_j u$ converges
  unconditionally in $\setBV(\Rn)$ together with 
  \begin{align}\label{eq:stab2}
    \|\mathcal{B}_\lambda u\|_{q}^{q} \leq
    c\,\int_{\mathcal{O}_{\lambda}}|u|^{q}\dif
    x\hspace{0.5cm}\text{and}\hspace{0.5cm} 
    |D (\mathcal{B}_\lambda u)|(\R^n) \leq c\,|Du|(\mathcal{O}_{\lambda}).
  \end{align}
  As a consequence, we have 
  \begin{align}\label{eq:stab3}
    \|T_\lambda u\|_{q} &\leq c\,\norm{u}_{q} \qquad
                                \text{and} \qquad
                                |D (T_\lambda u)|(\R^n) \leq c\,|Du|(\R^n).
  \end{align}
\end{lemma}
\begin{proof}
  Recall that $\bv(\Rn)\hookrightarrow\lebe^{\frac{n}{n-1}}(\Rn)$.
  Since, 
  $\support(\varphi_{j}*(\eta_{j}(u-u_{j})))\subset Q_{j}$, we
  directly find by Young's convolution inequality:
  \begin{align*}
    \int_{\R^n}|\mathcal{B}_{j}u|^{q}\dif x & = \int_{\frac 34 Q_j}|\eta_j (u-u_j) - \phi_j * (\eta_j (u-u_j)) |^{q}\dif x \leq\,c\,\int_{Q_{j}}|u|^{q}\dif x .
  \end{align*}
  Moreover, for each $j\in\mathbb{N}$ we obtain 
  \begin{align*}
    \abs{D(\mathcal{B}_j u)}(\R^n) 
    &= \bigabs{ D\big( \eta_j (u-u_j) -
      \phi_j * (\eta_j (u-u_j))
      \big)}(\R^n)
    \\
    &\leq 2\,\abs{ D( \eta_j (u-u_j))}(\R^n)
    \\
    &\leq c\,\abs{Du}(\tfrac 34 Q_j) + c\, \int_{Q_j}
      \frac{\abs{u-u_j}}{r_j} \dif x
    \\
    &\leq c\, \abs{Du}(Q_j),
  \end{align*}
  where we used Poincar\'{e}'s inequality in the last step. This yields \eqref{eq:stab1}. Now, \eqref{eq:stab2} follows by summing over $j$ and using the   finite intersection property of the $Q_j$, cf.~\ref{W5}. Finally, \eqref{eq:stab3} is a straightforward consequence of \eqref{eq:lipdef} and \eqref{eq:stab2}. The proof is complete. 
\end{proof}

We will now show that $T_\lambda u$ is in fact a Lipschitz continuous
function.
\begin{lemma}
  \label{lem:lip}
  There exists a constant $c=c(n)>0$ such that for all $\lambda>0$ there holds
  \begin{align}
    \label{eq:keyest}
    \mathcal{M} (D T_\lambda u) &\leq
                                  c\,\frac{\lambda}{h(\lambda)^{n+1}}.
  \end{align}
\end{lemma}
\begin{proof}
  Let $Q$ be an open cube with side length~$r$. We use the alternatives
  of Lemma~\ref{lem:alternatives}. We begin with alternative~\emph{\ref{itm:alt_small}}.  In this case, there
  exists $k \in \setN$ such that $Q \cap \frac 12 Q_k \neq \emptyset$,
  $8 r \leq r_k$ and $Q \subset \frac 34 Q_k$.  Then
  $T_\lambda u 
  = \sum_{j \in A_k} (\eta_j u_j +\phi_j * (\eta_j
  (u-u_j)) )$ 
  on $Q$ by \ref{W2} and therefore
  \begin{align*}
    D T_\lambda u 
    &= \sum_{j\in A_k} \nabla\Big(\eta_j (u_j-u_k) +\phi_j * (\eta_j
      (u-u_j)) \Big) 
    \\
    &=\sum_{j\in A_k}\Big(\nabla\eta_j (u_j-u_k) \Big)
      +\sum_{j\in A_k} \nabla\phi_j *\big( \eta_j
      (u-u_j) \big).
  \end{align*}
  Thus, by \ref{itm:P3} and $\mathcal{M}_{Q}f \leq |f|$ for all $f\in\lebe^{\infty}(\R^{n})$,
  \begin{align*}
    \mathcal{M}_Q(D T_\lambda u) 
    &\leq
      \sum_{j\in A_k}\mathcal{M}_Q\big((u_j-u_k) \nabla\eta_{j} \big)
      +\sum_{j\in A_k} \mathcal{M}_Q\Big(\nabla \phi_j *\big(\eta_j
      (u-u_j) \big)\Big)
    \\
    &\leq
      \sum_{j\in A_k}\mathcal{M}_Q\big((u_j-u_k) \nabla\eta_j\big)
      +\sum_{j\in A_k} \bignorm{\nabla \phi_j *\big(\eta_j
      (u-u_j) \big)}_\infty
    \\
    &\leq c\,
      \sum_{j\in A_k} \frac{\abs{u_j-u_k}}{r_k}
      +
      \sum_{j\in A_k} \norm{\nabla \phi_j}_\infty \int_{Q_j} \abs{u-u_j}\,dx
    \\
    &\leq c\,
      \sum_{j\in A_k} \frac{\abs{u_j-u_k}}{r_k}
      +
      c \sum_{j\in A_k} \epsilon_j^{-n-1} \int_{Q_j} \abs{u-u_j}\,dx
  \end{align*}
  using that $\support(\eta_j) \subset \frac 34Q_j$, $\epsilon_j \leq
  \frac 14 r_j$, and the properties of~$\phi_j$. Now,
  Lemma~\ref{lem:1}, $\epsilon_j = h(\lambda) \frac 14 r_j$ and
  $h(\lambda) \leq 1$ imply
  \begin{align}\label{eq:keyestsemi}
    \mathcal{M}_Q(D T_\lambda u) &\leq c\, \lambda h(\lambda)^{-n-1}.
  \end{align}
  We turn to alternative~\emph{\ref{itm:alt_large}}. In particular, for all
  $j \in\setN$ with $Q \cap \frac 34 Q_j \neq \emptyset$, there holds
  $r_j \leq 16 r$ and $\abs{Q_j} \leq 8^n \abs{Q_j \cap Q}$.
  Moreover, $137 Q \cap (\Rn \setminus \mathcal{O}_\lambda) \neq \emptyset$.
  Recall that $T_\lambda u = u - \sum_j \mathcal{B}_j u$ with
  convergence of the sum in the norm topology on $\setBV(\Rn)$, see
  Lemma~\ref{lem:Bj-stability}. Thus,
  \begin{align}\label{eq:Lipschitzsemi1}
    \mathcal{M}_Q(D T_\lambda u) &\leq \mathcal{M}_Q(D u) +
                                   \sum_{j\,:\, Q \cap \frac 34 Q_j \neq \emptyset}
                                   \mathcal{M}_Q(D \mathcal{B}_j u).
  \end{align}
  We address the estimation of the single terms $\mathcal{M}_{Q}(D\mathcal{B}_{j}u)$ first. We start by noting that for any $v\in\sobo^{1,1}(\R^{n})$ with support in $\frac{3}{4}Q_{j}$ there holds 
  \begin{align*}
    \mathcal{M}_{Q_{j}}\big(\phi_j * Dv \big)& \leq \frac{1}{|Q_{j}|}\int_{Q_{j}}\int_{\R^{n}}|\varphi_{j}(y)|\,|Dv(x-y)|\dif y\dif x\\
                                             & = \frac{1}{|Q_{j}|}\int_{\R^{n}}|\varphi_{j}(y)|\,\int_{Q_{j}}|Dv(x-y)|\dif x\dif y \\ 
                                             & \leq \dashint_{Q_{j}}|Dv|\dif x = \mathcal{M}_{Q_{j}}(Dv)
  \end{align*}
  since $\support(\eta_{j})+\support(\varphi_{j})\Subset Q_{j}$. If $v\in\setBV(\R^{n})$ has support in $\frac{3}{4}Q_{j}$, choose a sequence $(v_{k})\subset\sobo^{1,1}(\R^{n})$ such that $\support(v_{k})\subset\frac{3}{4}Q_{j}$ and $v_{k}\to v$ strictly in $\setBV(\R^{n})$. Clearly, $\varphi_{j}*v_{k}\to\varphi_{j}*v$ in $\lebe_{\locc}^{1}(\R^{n})$ and since $D(\varphi_{j}*v)=\varphi_{j}*Dv$, lower semicontinuity of the total variation with respect to $\lebe_{\locc}^{1}$-convergence implies 
  \begin{align*}
    \mathcal{M}_{Q_{j}}(\varphi_{j}*Dv)
    & \leq
      \mathcal{M}_{Q_{j}}(D(\varphi_{j}*v))
    \\ 
    & \leq
      \liminf_{k\to\infty}\mathcal{M}_{Q_{j}}(D(\varphi_{j}*v_{k}))
    \\
    & = \liminf_{k\to\infty}\mathcal{M}_{Q_{j}}(\varphi_{j}*Dv_{k}) \leq  \liminf_{k\to\infty}\mathcal{M}_{Q_{j}}(Dv_{k}) \leq \mathcal{M}_{Q_{j}}(Dv)
  \end{align*}
  provided $\support(Dv)$ is a closed subset of $Q_{j}$. Applying the previous inequality to $v=\eta_{j}(u-u_{j})$, we estimate 
  \begin{align}\label{eq:Lipschitzsemi2}
    \begin{split}
      \mathcal{M}_Q(D \mathcal{B}_j u) 
      &\leq \mathcal{M}_Q\big(  D(\eta_j
      (u-u_j)\big) + \mathcal{M}_{Q}\big(\phi_j * D(\eta_j 
      (u-u_j)) \big)
      \\
      &\leq \Big( \frac{|Q_{j}|}{|Q|}\Big(\mathcal{M}_{Q_{j}}\big(  D(\eta_j
      (u-u_j))\big) + \mathcal{M}_{Q_{j}}\big(\phi_j * D(\eta_j 
      (u-u_j)) \big)\Big) \\ 
      & \leq  c\Big( \frac{|Q_{j}\cap Q|}{|Q|}\Big(\mathcal{M}_{Q_{j}}\big(  D(\eta_j
      (u-u_j))\big)\Big), 
    \end{split}
  \end{align}
  the geometric alternative \emph{\ref{itm:alt_large}} having entered in the last step only. By Lemma~\ref{lem:1}\ref{item:aux1}, we thus obtain 
  \begin{align}\label{eq:Lipschitzsemi3}
    \mathcal{M}_{Q_j} \big(  D(\eta_j
    (u-u_j)\big) &\le c\, \dashint_{Q_j} \frac{\abs{u-u_j}}{r_j}\,dx + c\,
                   \mathcal{M}_{Q_j} (Du)  \leq c\, \lambda.
  \end{align}
  Thus, combining \eqref{eq:Lipschitzsemi1}, \eqref{eq:Lipschitzsemi2} and \eqref{eq:Lipschitzsemi3}, \emph{\ref{itm:alt_large}} and the finite intersection of the $Q_{j}$'s, cf.~\ref{W5}, imply
  \begin{align*}
    \mathcal{M}_Q(\nabla T_\lambda u)
    &\leq \mathcal{M}_{137 Q}(Du) + c \!\!\!\sum_{j\,:\, Q \cap \frac
      34 Q_j \neq \emptyset} \frac{\abs{Q_j\cap Q}}{\abs{Q}} \lambda
      \leq c\,\lambda.
  \end{align*}
  Recalling \eqref{eq:keyestsemi} and that $h\colon (0,\infty)\to (0,1]$, the proof is hereby complete. 
\end{proof}
The following corollary justifies the name \emph{Lipschitz truncation}.
\begin{corollary}
  \label{cor:lip}
  For each $\lambda>0$ we have $T_\lambda u \in W^{1,\infty}(\Rn)$. More precisely, there exists $c=c(n)>0$ such that for all $\lambda>0$ there holds $\norm{\nabla T_\lambda u}_\infty \leq
  c\,\frac{\lambda}{h(\lambda)^{n+1}}$.
\end{corollary}
\begin{proof}
  This is a direct consequence of Lemma~\ref{lem:lip} and Lemma~\ref{lem:franz}\ref{item:MAX3}.
\end{proof}
We now turn to the convergence properties of $T_{\lambda}u \to u$ as
$\lambda\to\infty$.  The core feature of our truncation operator
$T_\lambda$ is that it possesses a nice (almost) dual operator
$S_\lambda$ which satisfies
$DT_\lambda \approx S_\lambda^* D$,
see~\eqref{eq:commutator-type}. Let us define for
$\rho\in \hold_c(\Omega;\R^{n})$
\begin{align}
  \label{eq:def-Slambda}
  S_\lambda\rho := \rho - \sum_{j}\eta_{j}(\rho-\varphi_{j}*\rho)=
  \rho \indicator_{\mathcal{O}_\lambda^\complement} + \sum_j \eta_j (\phi_j * \rho).
\end{align}
\begin{lemma}[Commutator type estimate]
  \label{lem:Slambda}
  The operator~$S_\lambda$ as given in \eqref{eq:def-Slambda} satisfies the following:
  \begin{enumerate}
  \item \label{itm:Slambda_Linfty} $S_\lambda$ is non-expansive for the $\lebe^{\infty}$-norm in the sense that for all
    $\rho\in \hold_c(\Omega;\R^{n})$ there holds
    \begin{align*}
      \norm{S_\lambda \rho}_\infty \leq \norm{\rho}_\infty,
    \end{align*}
  \item \label{itm:Slambda_commutator} For all $\rho\in\hold_{c}(\Omega;\R^{n})$ and $u\in\setBV(\R^{n})$ we have the \emph{commutator-type estimate}
    \begin{align}
      \label{eq:commutator-type}
      \bigabs{\skp{D T_\lambda u}{\rho} - \skp{Du}{S_\lambda \rho}}
      &\leq c\, h(\lambda) \abs{Du}(\mathcal{O}_\lambda)\, \norm{\rho}_\infty.
    \end{align}
  \end{enumerate}
\end{lemma}
\begin{proof} 
  The claim of~\ref{itm:Slambda_Linfty} follows by the pointwise estimate
  \begin{align*}
    |S\rho| & \leq
              |\rho|\indicator_{\mathcal{O}_{\lambda}^\complement}+\sum_{j}\eta_{j}|\varphi_{j}*\rho|
              \leq \norm{\rho}_\infty\indicator_{\mathcal{O}_{\lambda}^\complement} +
              \sum_j \eta_j \norm{\rho}_\infty \leq \norm{\rho}_\infty.
  \end{align*}
  Let us turn to the proof of~\ref{itm:Slambda_commutator}.  By a routine approximation argument, it suffices to consider
  $\rho \in \hold^1_c(\Omega;\R^{n})$ with $\norm{\rho}_\infty \leq 1$.
  Then
  \begin{align}
    \begin{aligned}
      \lefteqn{\skp{D T_\lambda u}{\rho} - \skp{D u}{S_\lambda \rho}}
      \quad&
      \\
      &= - \skp{T_\lambda u}{\divergence(\rho)} - \skp{D u}{S_\lambda \rho}
      \\
      & = - \Bigskp{ u - \sum_j \Big(\eta_j (u-u_j) - \phi_j * (\eta_j
        (u-u_j)) \Big)}{\divergence \rho}
      \\
      &\quad - \skp{Du}{\rho -
        \sum_{j}\eta_{j}(\rho-\varphi_{j}*\rho)}
      \\
      & = -\Bigskp{ \sum_j D \Big(\eta_j (u\!-\!u_j) - \phi_j * (\eta_j
        (u\!-\!u_j)) \Big)}{\rho} +
      \skp{Du}{\sum_{j}\eta_{j}(\rho\!-\!\varphi_{j}*\rho)}
      \\
      & = -\Bigskp{ \sum_j \Big(\eta_j Du - \phi_j * (\eta_j
        Du) \Big)}{\rho}  +
      \skp{Du}{\sum_{j}\eta_{j}(\rho-\varphi_{j}*\rho)}
      \\
      &\quad -\Bigskp{ \sum_j \Big(\nabla \eta_j (u-u_j) - \phi_j *
        (\nabla \eta_j
        (u-u_j)) \Big)}{\rho}
      \\
      &= -\Bigskp{ \sum_j \Big(\nabla \eta_j (u-u_j) - \phi_j *
        (\nabla \eta_j
        (u-u_j)) \Big)}{\rho}.
    \end{aligned}
  \end{align}
  In particular,
  \begin{align*}
    \bigabs{\skp{D T_\lambda u}{\rho} - \skp{D u}{S_\lambda \rho}}
    &\leq
      \sum_{j}\int_{\frac 34 Q_j}\left\vert
      ((u-u_{j})\nabla\eta_{j})-\phi_{j}*((u-u_{j})\nabla\eta_{j})\right\vert \dif x.
  \end{align*}
  Now, we use the well known mollifier estimate 
  \begin{align}\label{eq:BVmollifierest}
    \norm{v - \phi_j * v}_1 \leq c\, 
    \epsilon_j \abs{D v}(\R^n).
  \end{align}
  Indeed, the $W^{1,1}$-case can be found in~\cite{MalZ97}, while the
  $\setBV$ case follows by approximation in the strict topology. Hence,
  \begin{align*}
    \bigabs{\skp{D T_\lambda u}{\rho} - \skp{D u}{S_\lambda \rho}}
    & \stackrel{\eqref{eq:BVmollifierest}}{\lesssim}
      \sum_{j}\epsilon_{j} \bigabs{D (\nabla\eta_{j}(u-u_{j}))}(Q_j)
    \\
    & \;\lesssim
      \sum_{j}\frac{\epsilon_{j}}{r_j}\Big(\int_{Q_j^*}\frac{|u-u_{j}|}{r_{j}}\dif
      x + |Du|(Q_j)\Big)
    \\
    & \;\leq h(\lambda) \sum_{j}|Du|(Q_j)
    \\
    & \;\leq h(\lambda) |Du|(\mathcal{O}_\lambda).
  \end{align*}
  This is \ref{itm:Slambda_commutator}, and the proof is complete. 
\end{proof}
We are now able to characterize to prove area-strict convergence.
\begin{lemma}[Area-strict convergence]\label{lem:BVstrict}
  \label{lem:area-strict}
  We have $T_\lambda u \to u$ in the area-strict sense of
  $\setBV(\R^{n})$ as $\lambda \to \infty$. In particular, $D T_\lambda
  u \to Du$ area strictly for~$\lambda \to \infty$. Moreover,
  \begin{align}
    \label{eq:area-strict}
    \area{\D T_{\lambda}u}(\R^{n})
    &\leq \area{Du}(\Rn) + ch(\lambda)|Du|(\mathcal{O}_{\lambda}) +
      \frac{c}{\lambda}\abs{\D u}(\Rn),
      \\
    \label{eq:area-strict2}
    \abs{\D T_{\lambda}u}(\R^{n})
    &\leq \abs{Du}(\Rn) + ch(\lambda)|Du|(\mathcal{O}_{\lambda}).
  \end{align}
\end{lemma}
\begin{proof}
  We start with the~$L^1$ convergence. Lemma~\ref{lem:Bj-stability},
  $\setBV(\R^{n})\hookrightarrow\lebe^{\frac{n}{n-1}}(\R^{n})$
  and $|\mathcal{O}_{\lambda}|\leq \frac{c}{\lambda}|Du|(\R^{n})$
  (which follows from Lemma \ref{lem:franz} (b)) imply
  \begin{align*}
    \norm{u - T_\lambda u}_1 &= \norm{\mathcal{B}_\lambda u}_1 \leq                                c\int_{\mathcal{O}_\lambda} \abs{u}\dif x \leq c\|u\|_{\frac{n}{n-1}}\Big(\frac{|Du|(\R^{n})}{\lambda} \Big)^{\frac{1}{n}}\to 0,\qquad \lambda\to\infty.
  \end{align*}
  Next, recall that the area-strict convergence of $DT_\lambda u$ to
  $Du$ is equivalent to strict convergence of
  $(DT_{\lambda}u,\mathscr{L}^{n})$ to $(Du,\mathscr{L}^{n})$.
  To prove the latter, let $\rho_{1}\in\hold_{c}^{1}(\R^{n},\Rn)$ and $\rho_2 \in \hold^1_c(\R^{n})$ be such that $\sqrt{|\rho_{1}|^{2}+|\rho_{2}|^{2}}\leq 1$. We estimate 
  \begin{align*}
    \lefteqn{ \bigabs{\skp{(D T_\lambda u,\mathscr{L}^n)}{(\rho_{1},\rho_2)}}  
    =
    \bigabs{\skp{D T_\lambda u}{\rho_{1}} + \skp{\mathscr{L}^n}{\rho_2}
    } } \quad 
    &
    \\
    &=
      \bigabs{
      \skp{D u}{S_\lambda \rho_{1}} + 
      \big(\skp{D T_\lambda u}{\rho_{1}} - \skp{Du}{S_\lambda \rho_{1}} \big)
      + \skp{\mathscr{L}^n}{\rho_2}
      }
    \\
    &=
      \bigabs{
      \skp{(D u,\mathscr{L}^n)}{(S_\lambda \rho_{1},S_{\lambda}\rho_2)} + 
      \big(\skp{D T_\lambda u}{\rho_{1}} - \skp{Du}{S_\lambda \rho_{1}} \big)
      + \skp{\mathscr{L}^n}{\rho_2-S_\lambda \rho_2}
      }
    \\
    &\leq
      \bigabs{
      \skp{(D u,\mathscr{L}^n)}{(S_\lambda \rho_{1},S_{\lambda}\rho_2})} + \bigabs{ 
      \big(\skp{D T_\lambda u}{\rho_{1}} - \skp{Du}{S_\lambda \rho_{1}} \big)}
      + \bigabs{\skp{\mathscr{L}^n}{\rho_2-S_\lambda \rho_2}
      } \\ 
    & =: \mathrm{I}+\mathrm{II}+\mathrm{III}. 
  \end{align*}
  By Lemma~\ref{lem:Slambda}\ref{itm:Slambda_Linfty}, $S_{\lambda}$ is non-expansive for the $\lebe^{\infty}$-norm and thus 
  \begin{align*}
    |(S_{\lambda}\rho_{1},S_{\lambda}\rho_{2})|\leq \sqrt{\|\rho_{1}\|_{\infty}^{2}+\|\rho_{2}\|_{\infty}^{2}}\leq 1.
  \end{align*}
  Hence, $\mathrm{I}\leq \area{Du}(\R^{n})$. For $\mathrm{II}$, we utilise Lemma~\ref{lem:Slambda}\ref{itm:Slambda_commutator} to find 
  \begin{align*}
    \mathrm{II}\leq ch(\lambda)|\D u|(\mathcal{O}_{\lambda})\|\rho_{1}\|_{\infty}\leq ch(\lambda)|Du|(\mathcal{O}_{\lambda})\stackrel{\lambda\to\infty}{\longrightarrow} 0
  \end{align*} 
  using $h(\lambda)\rightarrow0$ and $|\D u|(\mathcal{O}_{\lambda})\leq |\D u|(\R^n)<\infty$.
  Ad~$\mathrm{III}$. Using
  $\rho_2 = S_\lambda \rho_2\leq 1$ on $\mathcal{O}_\lambda^\complement$,
  $\norm{S_\lambda \rho_2}_\infty \leq \norm{\rho_2}_\infty$ and $|\mathcal{O}_{\lambda}|\leq\frac{c}{\lambda}|Du|(\R^{n})$,
  \begin{align*}
    \bigabs{\skp{\mathscr{L}^n}{\rho_2-S_\lambda \rho_2} }
    &\leq
      2\,\norm{\rho_2}_\infty
      \abs{\mathcal{O}_\lambda} \leq \frac{c}{\lambda}\abs{\D u}(\Rn)\stackrel{\lambda\to\infty}{\longrightarrow}0.  
  \end{align*}
  In consequence, gathering the estimates for $\mathrm{I},\mathrm{II},\mathrm{III}$, 
  \begin{align}
    \label{eq:1}
    \area{\D T_{\lambda}u}(\R^{n})
    &\leq \area{Du}(\Rn) + ch(\lambda)|Du|(\mathcal{O}_{\lambda}) + \frac{c}{\lambda}\abs{\D u}(\Rn).
  \end{align}
  This proves,~\eqref{eq:area-strict}. The
  estimate~\eqref{eq:area-strict2} follows analogously without the use of~$\rho_2$.
  Hence,
  \begin{align}\label{eq:gather1}
    \limsup_{\lambda\to\infty}|\area{\D T_{\lambda}u}|(\R^{n}) \leq  \area{\D u}(\R^{n}). 
  \end{align}
  On the other hand, by the first part of the proof, $T_{\lambda}u\to
  u$ in $\lebe_{\locc}^{1}(\R^{n})$. Thus, by the $L^1$-lower
  semicontinuity~\eqref{eq:L1-lsc} we obtain
  \begin{align*}
    \area{\D u}(\R^{n})\leq \liminf_{\lambda\to\infty}\area{\D T_{\lambda}u}(\R^{n}). 
  \end{align*}
  In conjunction with \eqref{eq:gather1}, this yields $\lim_{\lambda\to\infty}\area{DT_{\lambda}u}(\R^{n})=\area{Du}(\R^{n})$ and the proof is complete. 
\end{proof}
We conclude by identifying the limits of the single constituents of $T_{\lambda}u$:
\begin{lemma}
  \label{lem:decompositionlimit}
  The following hold: 
  \begin{enumerate}
  \item \label{itm:decompositionlimit1}
    $\indicator_{\mathcal{O}_\lambda^\complement} \nabla T_\lambda u =
    \indicator_{\mathcal{O}_\lambda^\complement} \nabla u \to \nabla u$ in $\lebe^1(\R^n)$ as $\lambda\to\infty$.
  \item \label{itm:decompositionlimit2}  $\nabla T_\lambda\mathscr{L}^{n}\mres\mathcal{O}_{\lambda} \to D^s u$ in the sense of area-strict convergence of $\R^{n}$-valued Radon measures.
  \end{enumerate}
\end{lemma}
\begin{proof}
  Since $\abs{\mathcal{O}_\lambda} \to 0$ as $\lambda\to\infty$ and
  the approximate gradient satisfies $\nabla u \in \lebe^1(\R^{n})$,
  we have
  $\nabla u - \indicator_{\mathcal{O}_\lambda^\complement} \nabla u =
  \indicator_{\mathcal{O}_\lambda} \nabla u \to 0$ in
  $\lebe^1(\Rn)$. This
  proves~\ref{itm:decompositionlimit1}. Ad~\ref{itm:decompositionlimit2}. Let
  $\varphi\in\hold_{c}(\R^{n})$. By Lemma~\ref{lem:BVstrict} it
  follows that $DT_{\lambda}u \to Du$ in the weak* sense, so
  \begin{align*}
    \langle \indicator_{\mathcal{O}_{\lambda}^\complement}\nabla T_{\lambda}u,\varphi\rangle = \langle \nabla T_{\lambda}u,\varphi\rangle - \langle \indicator_{\mathcal{O}_{\lambda}}\nabla T_{\lambda}u,\varphi\rangle \to \langle Du-\nabla u\mathscr{L}^{n},\varphi\rangle = \langle D^{s}u,\varphi\rangle. 
  \end{align*}
  It thus remains to establish that $\area{DT_{\lambda}u}(\mathcal{O}_{\lambda})\to \area{D^{s}u}(\R^{n})$ as $\lambda\to\infty$. To this end, we record that 
  \begin{align*}
    \area{DT_{\lambda}u}(\mathcal{O}_{\lambda})
    & =
      \area{DT_{\lambda}u}(\R^{n})-\area{DT_{\lambda}u}(\mathcal{O}_{\lambda}^\complement)
    \\
    & = \area{DT_{\lambda}u}(\R^{n})-\area{Du}(\mathcal{O}_{\lambda}^\complement)
    \\ 
    & = \area{DT_{\lambda}u}(\R^{n})-(\area{\nabla u\mathscr{L}^{n}}(\R^{n}) -
      \area{\nabla u\mathscr{L}^{n}}(\mathcal{O}_{\lambda}))
    \\ 
    & \leq (\area{DT_{\lambda}u}(\R^{n})-\area{\nabla
      u\mathscr{L}^{n}}(\R^{n})) +  |\mathcal{O}_{\lambda}| + |\nabla
      u\mathscr{L}^{n}|(\mathcal{O}_{\lambda})
    \\ 
    & \to \area{Du}(\R^{n})-\area{\nabla u}(\R^{n}) = \area{D^{s}u}(\R^{n}),\qquad \lambda\to\infty, 
  \end{align*}
  where we have used that
  $\nabla u
  \mathscr{L}^{n}\mres\mathcal{O}_{\lambda}^\complement=Du\mres\mathcal{O}_{\lambda}^\complement$
  in the third equality, the trivial bound
  $\sqrt{1+|\cdot|^{2}}\leq 1+|\cdot|$ in the fourth and
  $|\mathcal{O}_{\lambda}|\to 0$ in conjunction with
  \ref{itm:decompositionlimit1} in the ultimate line. This establishes
  $\limsup_{\lambda\to\infty}\area{DT_{\lambda}u}(\mathcal{O}_{\lambda})\leq
  \area{D^{s}u}(\R^{n})$. On the other hand, the $L^1$ lower
  semicontinuity~\eqref{eq:L1-lsc} implies
  $\area{Du}(\R^{n})\leq
  \liminf_{\lambda\to\infty}\area{DT_{\lambda}u}(\R^{n})$ so that, in
  total,
  $\area{Du}(\R^{n})=\lim_{\lambda\to\infty}\area{DT_{\lambda}u}(\R^{n})$. The
  proof is complete.
\end{proof}

\subsection{Preserving zero boundary values}
\label{sec:pres-zero-bound}

Sometimes it is desirable to preserve zero boundary values of a given function. We show in this section how to modify our Lipschitz truncation such that the~$u_\lambda$ also have zero boundary values.

Hence, let~$\Omega$ be a bounded Lipschitz domain and let~$u \in
\setBV(\Rn)$ with $u=0$ on $\Rn \setminus \Omega$. We take the same
decomposition of our bad set by a Whitney
cover as in the beginning of the section. Recall that
\begin{align*}
  \mathcal{B}_j u&= \eta_j (u-u_j) - \phi_j * (\eta_j
                   (u-u_j)),
  \\                   
  T_\lambda u &= u_\lambda = u - \sum_j \mathcal{B}_j u.
\end{align*}
To obtain~$T_\lambda u=0$ on~$\Rn \setminus \Omega$, we have
to ensure that~$\mathcal{B}_j u=0$ on~$\Rn \setminus \Omega$. For
this, let~$Q_j$ be a cube close to the boundary~$\partial \Omega$,
i.e.  $\frac 34 Q_j \not\subset \Omega$). In this case the definition
of the~$u_j$ in~\eqref{eq:defvj} ensures that~$u_j=0$. Thus, in this
case
\begin{align*}
  \mathcal{B}_j u &= \eta_j u - \phi_j *(\eta_j u).
\end{align*}
By assumption on~$u$, we have $\eta_j(u-u_j)= \eta_j u = 0 $
on~$\Rn \setminus \Omega$. However, the convolution with~$\phi_j$
might transport values of~$u$ to~$\Rn \setminus \Omega$. To avoid
this, it is necessary to use a directed convolution. So have to drop
the assumption that the~$\phi_j$ are radially symmetric mollifiers.

By Lemma~\ref{lem:franz}~\ref{item:MAX1}, we have
\begin{align*}
  \abs{Q_j} \leq \mathscr{L}^n(\set{\mathcal{M}(Du) > \lambda})
  \lesssim \frac{\abs{Du}(\Rn)}{\lambda}. 
\end{align*}
Thus, for large~$\lambda$ the Whitney cubes are small. Now, since~$\Omega$
is a Lipschitz domain, its boundary~$\partial \Omega$ can be written
locally on~$Q_j$ as a graph of a Lipschitz function. Thus, there
exists a unit vector~$\nu_j$ (an approximation of the normal
of~$\partial \Omega$ on~$Q_j$) such that $Q_j \cap \Omega$ satisfies
the outer cone condition in direction~$\nu_j$. Thus, we can
choose~$K=K(\Omega)\geq 1$ such that for all~$x \in \Omega$ we have
\begin{align}
  \label{eq:directed}
  \Omega^\complement + \ball_{\frac 1K}\Big(\tfrac 12 \nu_j\Big) = \Bigset{x+y\,:\,
  x\in \Omega^\complement, y \in \ball_{\frac 1K}\Big(\tfrac 12 \nu_j\Big)}  \subset
  \Omega^\complement. 
\end{align}
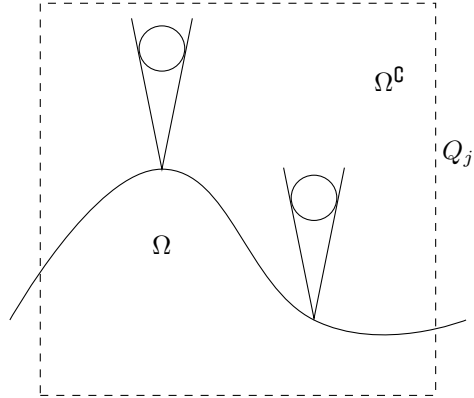
\begin{figure}[ht!]
  \centering
  \begin{tikzpicture}[scale=2]
    \draw plot [smooth, tension=0.8] coordinates { (-1,-1) (0,0)
      (1,-1) (2,-1)};
    \begin{scope}[shift={(0,0)}]
      \draw (-0.2,1) -- (0,0) -- (0.2,1); \draw (0,0.8) circle
      (1.5mm);
    \end{scope}
    \begin{scope}[shift={(1,-0.99)}]
      \draw (-0.2,1) -- (0,0) -- (0.2,1); \draw (0,0.8) circle
      (1.5mm);
    \end{scope}
    \node at (1.5,0.6) {$\Omega^\complement$}; \node at (1.95,0.1)
    {$Q_j$}; \node at (0.0,-0.5) {$\Omega$}; \draw[dashed] (-0.8,-1.5)
    rectangle (1.8,1.1);
  \end{tikzpicture}
  
  \caption{Local chart of the boundary. The cones indicate the
    direction of the convolution.}
  \label{fig:cone}
\end{figure}
Now, let $\phi$ be a smooth, non-negative, radially symmetric
mollifier with support in the unit ball. Then we define the local
directed mollifier~$\phi_j$ by
\begin{align*}
  \phi_j(x) &:= (K\, \epsilon_j)^{-n}
              \phi\bigg(\frac{x}{K\, \epsilon_j}+\frac{\nu_j}{2}\bigg)
              \qquad \text{with}
              \qquad \epsilon_j := h(\lambda)\, \tfrac 14
              \,r_j. 
\end{align*}
Then~\eqref{eq:directed} ensures that $\phi_j * (\eta_j u) =0$
on~$\Rn \setminus \Omega$. The same holds for the~$\mathcal{B}_j
u$. Consequently, $u_\lambda = 0$ on $\Rn \setminus \Omega$ provided
that~$\lambda$ is large enough depending on the
boundary~$\partial \Omega$. Note that since the~$\phi_j$ are no longer
radially symmetric, one has to replace~$\phi_j$ in the definition of
the (almost) dual operator~$S_\lambda$ by $\overline{\phi_j}$ with
$\overline{\phi_j}(x) := \phi_j(-x)$.

\subsection{Proof of Theorem~\ref{thm:main}}
\label{sec:proof-theor-refthm:m}

We are now in position to prove our main theorem.

For a given $\lambda>0$, define $u_{\lambda}:=T_{\lambda}u$ as
in~\eqref{eq:lipdef}.  The Lipschitz property~\ref{itm:thm-lip}
follows from Corollary~\ref{cor:lip}. The smallness of the
set~$\set{u\neq u_\lambda}$ from \ref{itm:thm-small} is an immediate
consequence of the construction of $T_{\lambda}u$ and
Lemma~\ref{lem:franz}~\ref{item:MAX2}. The stability asssertions
of~\ref{itm:thm-stab} are given in Lemma \ref{lem:Bj-stability}. On
the other hand, the convergence properites~\ref{itm:thm-conv} follow
from Lemma~\ref{lem:BVstrict} and Lemma~\ref{lem:decompositionlimit}.
The preservation of the zero boundary values~\ref{itm:thm-zero} follows
from Subsection~\ref{sec:pres-zero-bound}.  The proof of
Theorem~\ref{thm:main} is complete. \qed

\def\polhk#1{\setbox0=\hbox{#1}{\ooalign{\hidewidth
  \lower1.5ex\hbox{`}\hidewidth\crcr\unhbox0}}}
  \def\ocirc#1{\ifmmode\setbox0=\hbox{$#1$}\dimen0=\ht0 \advance\dimen0
  by1pt\rlap{\hbox to\wd0{\hss\raise\dimen0
  \hbox{\hskip.2em$\scriptscriptstyle\circ$}\hss}}#1\else {\accent"17 #1}\fi}
  \def\ocirc#1{\ifmmode\setbox0=\hbox{$#1$}\dimen0=\ht0 \advance\dimen0
  by1pt\rlap{\hbox to\wd0{\hss\raise\dimen0
  \hbox{\hskip.2em$\scriptscriptstyle\circ$}\hss}}#1\else {\accent"17 #1}\fi}
  \def\ocirc#1{\ifmmode\setbox0=\hbox{$#1$}\dimen0=\ht0 \advance\dimen0
  by1pt\rlap{\hbox to\wd0{\hss\raise\dimen0
  \hbox{\hskip.2em$\scriptscriptstyle\circ$}\hss}}#1\else {\accent"17 #1}\fi}
  \def\ocirc#1{\ifmmode\setbox0=\hbox{$#1$}\dimen0=\ht0 \advance\dimen0
  by1pt\rlap{\hbox to\wd0{\hss\raise\dimen0
  \hbox{\hskip.2em$\scriptscriptstyle\circ$}\hss}}#1\else {\accent"17 #1}\fi}
  \def\cprime{$'$}
\providecommand{\bysame}{\leavevmode\hbox to3em{\hrulefill}\thinspace}
\providecommand{\MR}{\relax\ifhmode\unskip\space\fi MR }
\providecommand{\MRhref}[2]{%
  \href{http://www.ams.org/mathscinet-getitem?mr=#1}{#2}
}
\providecommand{\href}[2]{#2}


\end{document}